\documentclass[11pt,a4paper]{amsart}
\usepackage{amssymb,amsmath,epsfig,graphics,mathrsfs}
\usepackage{graphicx}
\usepackage{cancel}
\usepackage{comment}

\usepackage{fancyhdr}
\usepackage{hyperref}
\hypersetup{
 colorlinks   = true,
 urlcolor     = blue,
 linkcolor    = blue,
 citecolor   = red ,
 bookmarksopen=true
}

\usepackage{amsmath}
\usepackage{amsfonts}
\usepackage{amssymb}
\usepackage{amsthm}
\usepackage{epsfig,graphics,mathrsfs}
\usepackage{graphicx}

\usepackage[usenames, dvipsnames]{color} 

\usepackage{hyperref}
\usepackage{float}
\usepackage{caption}
\def \phi {\varphi}

\def \RN {\mathbb{R}^N}
\def \R {\mathbb{R}}

\def \G{\Gamma}

\def \vf{\varphi}

\newcommand{\no}{\mathbb N_0}
\newcommand{\sa}{\langle}
\newcommand{\da}{\rangle}
\newcommand{\so}{\mathbb S^{n-1}}

\newcommand{\Rn}{\mathbb R^n}

\newcommand{\p}{\partial}

\newcommand{\la}{\lambda}

\newcommand*\MSC[1][1991]{\par\leavevmode\hbox{%
\textit{\,\,\,\,\, #1 Mathematical subject classification:\ }}}
\newcommand\blfootnote[1]{%
  \begingroup
  \renewcommand\thefootnote{}\footnote{#1}%
  \addtocounter{footnote}{-1}%
  \endgroup}

\numberwithin{equation}{section}

\newcommand{\beq}{\begin{equation}}
\newcommand{\bea}[1]{\begin{array}{#1} }
\newcommand{\eeq}{ \end{equation}}
\newcommand{\ea}{ \end{array}}

\newcommand{\ve}{\varepsilon}

\newcommand{\re}{R(\ve)}
\newtheorem{theorem}{Theorem}[section]
\newtheorem{lemma}[theorem]{Lemma}
\newtheorem{proposition}[theorem]{Proposition}
\newtheorem{corollary}[theorem]{Corollary}
\newtheorem{remark}[theorem]{Remark}
\newtheorem{definition}[theorem]{Definition}

\numberwithin{equation}{section}

\begin{document}

\title[Some inequalities for the Fourier transform, etc.]{Some inequalities for the Fourier transform and their limiting behaviour}

{\blfootnote{\MSC[2020]{42B10, 42B15, 33C10}}}

\keywords{Fourier transform. Riesz Measures. Nonlocal restriction inequality.}

\date{}

\begin{abstract}
We identify a one-parameter family of inequalities for the Fourier transform whose limiting case is the restriction conjecture for the sphere. Using Stein's method of complex interpolation we prove the conjectured inequalities when the target space is $L^2$,  and show that this recovers in the limit the celebrated Tomas-Stein theorem.  
\end{abstract}

%\begin{abstract} 
%\end{abstract}
%\maketitle
\author{Nicola Garofalo}

\address{Dipartimento d'Ingegneria Civile e Ambientale (DICEA)\\ Universit\`a di Padova\\ Via Marzolo, 9 - 35131 Padova,  Italy}
\vskip 0.2in
\email{nicola.garofalo@unipd.it}

\thanks{The author has been supported in part by a Progetto SID (Investimento Strategico di Dipartimento): ``Aspects of nonlocal operators via fine properties of heat kernels", University of Padova, 2022. He has also been partially supported by a Visiting Professorship at the Arizona State University.} 

\dedicatory{In ricordo di E. M. Stein}

\maketitle

\tableofcontents

\section{Introduction}

For any $0<s<1$ and $r>0$ we consider the function
\begin{equation}\label{ar}
A^{(s)}_r(x) =   \frac{c(n,s) r^{2s}}{(|x|^2 - r^2)_+^s |x|^n}, 
\end{equation}
where, using the standard notation $\sigma_{n-1} = \frac{2\pi^{\frac n2}}{\G(\frac n2)}$ for the $(n-1)$-dimensional measure of the unit sphere $\so \subset \Rn$,  we have let  
\begin{equation}\label{norm}
c(n,s) = \frac{2}{\G(s) \G(1-s) \sigma_{n-1}}.
\end{equation}
Define measures on $\Rn$ by letting
\begin{equation}\label{mus}
d\mu^{(s)}_r(x) = A^{(s)}_r(x) dx.
\end{equation}
Since from \eqref{ar} it immediately follows that $A^{(s)}_r \in L^1(\Rn)$, from  \cite[Lemma 15.3]{Gft}) one has for every $s\in (0,1)$ and $r>0$
\begin{equation}\label{L:uno}
\mu^{(s)}_r(\Rn) = ||A^{(s)}_r||_{L^1(\Rn)} = \int_{\Rn} A^{(s)}_r(x) dx = \int_{\Rn} A^{(s)}_1(x) dx = 1.
\end{equation}
We next define the operator
\begin{equation}\label{As}
\mathscr A^{(s)}_r f(x) = A^{(s)}_r \star f(x).
\end{equation}
It follows from \eqref{L:uno} and Young's convolution theorem that  
\[
\mathscr A^{(s)}_r :  L^p(\Rn)\ \longrightarrow\ L^p(\Rn),\ \ \ \ \ \ \ \ 1\le p \le \infty,
\]
and that for any $f \in L^p(\Rn)$
we have 
\begin{equation}\label{p}
||\mathscr A^{(s)}_r  f||_{L^p(\Rn)} \le  ||f||_{L^p(\Rn)},
\end{equation}
which shows that $\mathscr A^{(s)}_r$ is a contraction in $L^p(\Rn)$ for every $0<s<1$ and $r>0$.

\medskip

For any $f\in \mathscr S(\Rn)$ we denote by $\hat f(\xi) = \mathscr F(f)(\xi) = \int_{\Rn} e^{-2\pi i \sa\xi,x\da} f(x) dx$ its Fourier transform, and ask the following

\vskip 0.2in

\noindent \textbf{Question:} Let $\frac 12\le s<1$. If $1\le p < \frac{2n}{n+2s -1}$, is it true that for $1\le q \le \frac{n+1-2s}{n+1} p'$ and for every $f\in \mathscr S(\Rn)$ one has for some $C^{(s)}(n,p)>0$
\begin{equation}\label{restrs}
\left(\int_{\Rn} |\hat f(\xi)|^{q} A^{(s)}_1(\xi) d\xi\right)^{1/q} \le C^{(s)}(n,p)\ ||f||_{L^{p}(\Rn)}.
\end{equation} 

\vskip 0.2in

We note that the restriction $s\ge \frac 12$ serves to guarantee that $\frac{2n}{n+2s -1}\le 2$. Therefore, the hypothesis $f\in L^p(\Rn)$ implies that $f$ be in the Hausdorff-Young range $[1,2]$. As a consequence, $\hat f$ is a function in $L^{p'}(\Rn)$.

Our interest in the above conjecture stems from the following observations. Assume that \eqref{restrs} does hold for any $s$ such that $\frac 12 \le s < 1$. Since $\frac{2n}{n+2s -1} \searrow \frac{2n}{n +1}$ as $s\nearrow 1$, if we take $1\le p < \frac{2n}{n+1}$ and $q\le \frac{n-1}{n+1} p'$, then it is immediate to verify that for any $s\in [\frac 12,1)$
\[
1\le p < \frac{2n}{n+2s -1}\  \ \ \ \text{and}\ \ \ \ 1\le q \le \frac{n+1-2s}{n+1} p'
\]
(for the second of these inequalities simply note that $q\le \frac{n-1}{n+1} p' <  \frac{n-1+ 2(1-s)}{n+1} p' =\frac{n+1-2s}{n+1} p'$), therefore \eqref{restrs} holds. But we now have the following fact implicitly contained in the seminal work \cite{R} of M. Riesz (for a proof see \cite[Prop. 15.4]{Gft}).

\begin{proposition}\label{P:vague}
For every function $f\in \mathscr S(\Rn)$ and $x\in \Rn$, one has 
\[
\underset{s\to 1}{\lim}\ \mathscr A^{(s)}_r f(x) = \mathscr M_r(f,x),
\]
where
\[
\mathscr M_r(f,x) = \frac{1}{\sigma_{n-1} r^{n-1}} \int_{S(x,r)} f(y) d\sigma(y)
\]
is the spherical average of $f$ over the sphere $S(x,r) = \{y\in \Rn\mid|y-x|=r\}$.
\end{proposition}

Therefore, passing to the limit in \eqref{restrs} and using Proposition \ref{P:vague}, if $C^{(s)}(n,p)$ converges to a number $C(n,p)>0$ as $s\to 1$, we would infer that for $1\le p < \frac{2n}{n+1}$ and $q\le \frac{n-1}{n+1} p'$ the following limiting inequality holds   
\begin{equation}\label{restr}
\left(\frac{1}{\sigma_{n-1}}\int_{\so} |\hat f(\xi)|^q d\sigma(\xi)\right)^{1/q} \le C(n,p)\ ||f||_{L^p(\Rn)}.
\end{equation}
As it is well-known, this is the famous restriction conjecture of C. Fefferman and E. Stein for the Fourier transform, see \cite{Fe}, \cite{Femult}, \cite{Feisr}, \cite{SHA}, \cite{DC} and \cite{Tao}. 

\begin{figure}[H]
	\centering
	\includegraphics[width=130pt]{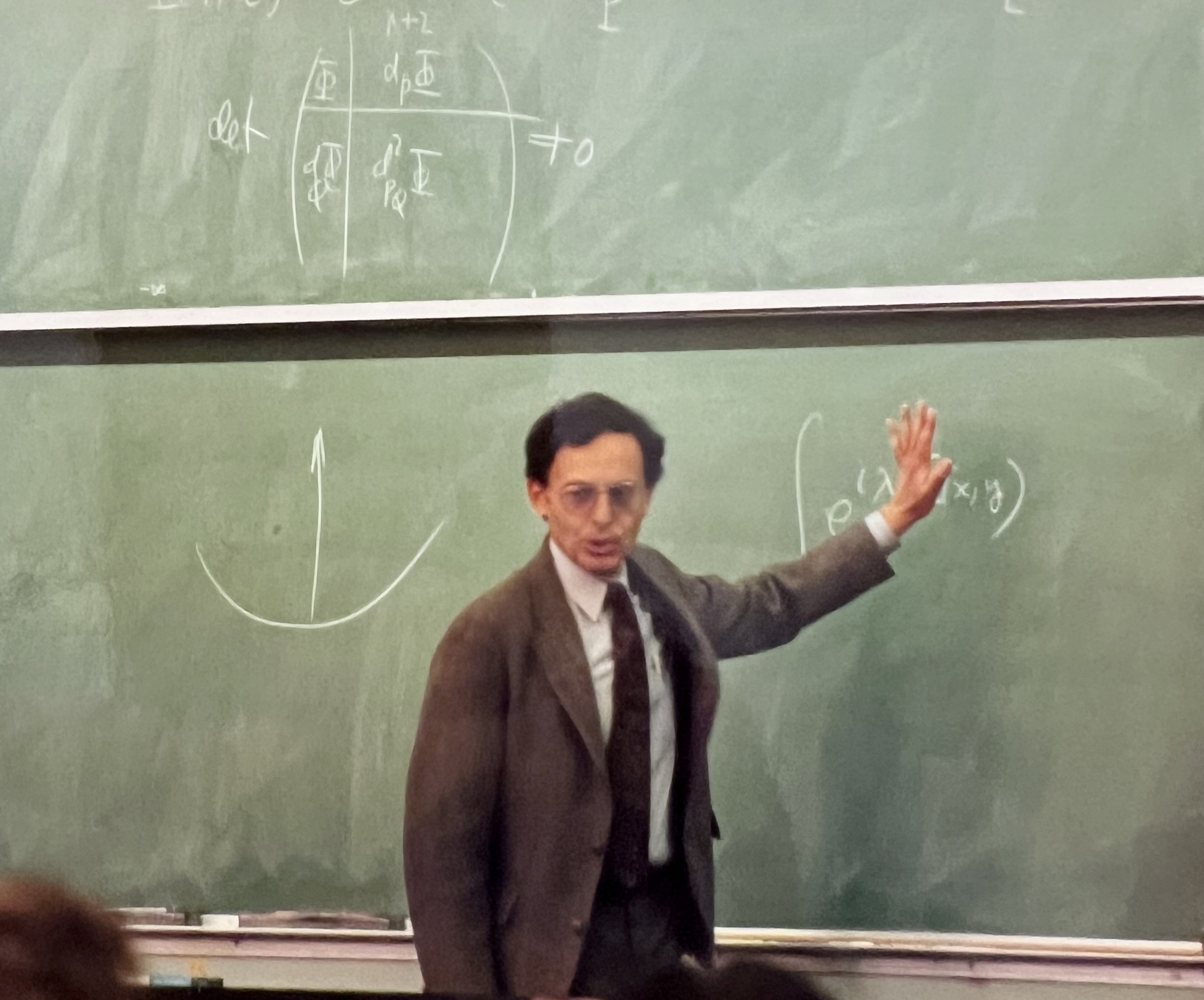}
	\caption*{Eli Stein lecturing on the restriction problem, UChicago, 1985}
	\label{fig:image1}
\end{figure}

One obvious advantage of \eqref{restrs} over \eqref{restr} is that the support of the measures \eqref{mus} is $\Rn\setminus B(0,1)$, instead of the 
lower dimensional manifold $\so\subset \Rn$. Notice that if we let $d\sigma$ denote the surface measure concentrated on the sphere $\so$, then Proposition \ref{P:vague} can be equivalently stated as follows
\begin{equation}\label{limS}
d\mu^{(s)}_1\ \underset{s\to 1}{\longrightarrow}\ \frac{1}{\sigma_{n-1}} d\sigma\ \ \ \ \ \ \text{in}\ \mathscr S'(\Rn).
\end{equation}

Concerning the measures \eqref{mus} we recall that in his above quoted paper M. Riesz developed the theory of the nonlocal operators $(-\Delta)^s$ and their inverses $I_{2s}$, the operators of fractional integration which play a pervasive role in analysis, see also \cite[Ch. 5]{St}. Among other things, he solved by inversion the Dirichlet problem 
\begin{equation}\label{dpnl}
\begin{cases}
(-\Delta)^s u = 0\ \ \  \text{in}\ B_r, 
\\
u = f\ \ \ \ \ \ \ \  \text{in}\ \Rn\setminus B_r,
\end{cases}
\end{equation}
and proved that for every $x\in B_r$ the unique solution to \eqref{dpnl} is provided by
\begin{equation}\label{u}
u(x) = c(n,s) \int_{|y|>r} \left(\frac{r^2 - |x|^2}{|y|^2 - r^2}\right)^s \frac{f(y)}{|y-x|^n}  dy,
\end{equation}
see formula (3) on p. 17 in \cite{R}, but also (1.6.11') and (1.6.2) on pages 122 and 112 in \cite{La}. It is clear from \eqref{u} that $u(0) = \int_{\Rn}  f(y) A^{(s)}_r(y) dy = \mathscr A^{(s)}_r f(0)$. The role of the operators $\mathscr A^{(s)}_r$ is further elucidated by the following result, see \cite[Prop. 15.6]{Gft}.
 
\begin{proposition}[The Blaschke-Privalov fractional Laplacian]\label{P:bps} Let $0<s<1$ and suppose that $f\in \mathscr L_s(\Rn)$ be in $C^{2s+\varepsilon}$ in a  neighborhood of $x\in \Rn$, for some $0<\varepsilon < 1$. One has
\begin{equation}\label{bp}
(-\Delta)^s f(x) = - \gamma(n,s) \underset{r\to 0^+}{\lim} \frac{\mathscr A^{(s)}_r f(x) -f(x)}{r^{2s}},
\end{equation}
where $\gamma(n,s) = \frac{s 2^{2s} \G(\frac n2 + s)}{\pi^{\frac n2} \G(1-s)}$.
\end{proposition}
Here, for $0<s<1$ we have denoted by $\mathscr L_s(\Rn)$ the space of measurable functions $f:\Rn\to \overline \R$ for which the norm
\[
||u||_{\mathscr L_s(\Rn)} = \int_{\Rn} \frac{|f(x)|}{1+|x|^{n+2s}} dx < \infty.
\]

\medskip

Returning to the inequality \eqref{restrs} we mention that, similarly to the restriction problem \eqref{restr}, it  cannot possibly hold for every exponent $p\in [1,2]$ in the Hausdorff-Young range. To understand the constraint $1\le p < \frac{2n}{n+2s -1}$, denote by $T:L^p(\Rn)\to L^q(\Rn,d\mu^{(s)}_1)$ the  ``restriction" operator in \eqref{restrs}. Then its adjoint $T^\star:L^{q'}(\Rn,d\mu^{(s)}_1)\to L^{p'}(\Rn)$ is easily seen to be given by 
\begin{equation}\label{Tstar}
T^\star f(\xi) = \int_{\Rn} e^{-2\pi i \langle\xi,x\rangle} f(x) A^{(s)}_1(\xi) dx, 
\end{equation}
so that
\[
T^\star 1 = \widehat{A^{(s)}_1}.
\]
Assuming \eqref{restrs}, we would thus have by duality  
\[
||T^\star 1||_{L^{p'}(\Rn)} \le B^{(s)}(n,p)\ ||1||_{L^{q'}(\Rn,d\mu^{(s)}_1)} = B^{(s)}(n,p) < \infty,
\] 
where in the last equality we have used  \eqref{L:uno}. Therefore, the validity of \eqref{restrs} for some $p$ implies that $\widehat{A^{(s)}_1}\in L^{p'}(\Rn)$. Now, in Section \ref{S:proof} we prove the following.

\begin{theorem}\label{T:friesz}
Let $s\in (0,1)$ and $n\ge 2$. Then the Fourier transform of the kernel defined by \eqref{ar} is given by      
\begin{equation}\label{ftAs}
\widehat{A^{(s)}_r}(\xi) =  \frac{2^{\frac n2-s} \G(\frac n2)}{\G(s)}\int_{2\pi r |\xi|}^\infty t^{s - \frac n2} J_{\frac n2 - 1 + s}(t) dt.
\end{equation}
\end{theorem}

Using Theorem \ref{T:friesz}, in \eqref{infty2} of Corollary \ref{C:as} we obtain the following important decay at infinity
\begin{equation}\label{decay}
\widehat{A^{(s)}_1}(\xi) \cong \frac{1}{|\xi|^{\frac{n+1}2 - s}},
\end{equation}
which shows that 
\begin{equation}\label{thick}
\widehat{A^{(s)}_1} \in L^{p'}(\Rn)\ \Longleftrightarrow\ p'> \frac{2n}{n+1-2s}\ \Longleftrightarrow\ 1\le p < \frac{2n}{n+2s -1}.
\end{equation}
We conclude that the inequality \eqref{restrs} cannot possibly hold for $p\ge \frac{2n}{n+2s -1}$.

Before proceeding we pause to comment on an aspect of Theorem \ref{T:friesz}. As the reader will see its proof is somewhat more involved than its well-known counterpart in the case $s=1$. This is due to the \emph{nonlocal} nature of the measure $d\mu^{(s)}_1$, compared to the  surface measure $d\sigma$ of the unit sphere $\so$. To explain this comment we stress that 
\begin{equation}\label{FTsphere}
\widehat{d\sigma}(\xi) = \frac{2\pi}{|\xi|^{\frac n2-1}} J_{\frac n2-1}(2\pi |\xi|)
\end{equation}  
is just a rescaled spherically symmetric eigenfunction of the (local) differential operator $\Delta$ in $\Rn$. To see this, simply observe that for every $\la >0$ the function
\[
f_\la(\xi) = \int_{\so} e^{i \sqrt \la\sa\xi,\omega\da} d\sigma(\omega)
\]  
is a spherically symmetric solution of the Helmholtz equation $\Delta f_\la = - \la f_\la$ in $\Rn$, which shows that \eqref{FTsphere} solves such PDE with $\la = 4\pi^2$. However, the equations \eqref{dpnl}, \eqref{u} above, and Proposition \ref{P:bps} in particular, underscore that the Fourier transform of the measure $d\mu^{(s)}_r$ is instead connected to the pseudodifferential operator $(-\Delta)^s$, and computations with such nonlocal operator are usually more involved. In this regard, we recall the following quote from p. 51 in \cite{La}:...``\emph{In the theory of M. Riesz kernels, the role of the Laplace operator, which has a local character, is taken...by a non-local integral operator...This circumstance often substantially complicates the theory...}". In Lemma \ref{L:limit} we show that for $n\ge 2$ one has for every $\xi\in \Rn$ 
\[
\underset{s\to 1}{\lim}\ \widehat{A^{(s)}_1}(\xi) = \frac{1}{\sigma_{n-1}} \widehat{d\sigma}(\xi). 
\]
This result (which can be seen as a comforting a posteriori confirmation of the correctness of the computations leading to Theorem \ref{T:friesz}) is of course not surprising in view of \eqref{limS} above. It is worthwhile noting at this moment that \eqref{ar} are clearly reminiscent of the classical Bochner-Riesz kernels
\begin{equation}\label{br}
K_z(x) = \frac{\left(1-|x|^2\right)^z_+}{\Gamma(z+1)},\ \ \ \ \Re
z>-1,
\end{equation} 
whose Fourier transform is given by
\begin{equation}\label{FTBR}
\hat
K_z(\xi)=\pi^{-z}|\xi|^{-\left(\frac{n}2+z\right)}J_{\frac{n}2+z}(2\pi|\xi|),\
\ \ \ \xi\in \Rn.
\end{equation}
If we compare this formula with \eqref{FTsphere} it is clear that $\hat K_z\to \frac 12 \  \widehat{d\sigma}$ as $z\to -1$. While the computation of the Fourier transform of $A^{(s)}_1$ is more involved than \eqref{FTBR}, there are advantages in working with \eqref{ar} instead of \eqref{br}. One of them is that, as we have mentioned, $\widehat{A^{(s)}_1}$ is directly connected to the eigenfunctions of the nonlocal operator $(-\Delta)^s$ (to be further analysed in a future work), while this is not the case for \eqref{FTBR}.

\medskip

We have seen that \eqref{restrs} is only possible when $p$ satisfies \eqref{thick}. But, given a $p$ in such range what is the optimal range of $q$'s? To answer this question we use the well-known argument of Knapp, except that because of the presence of the measure $d\mu_1^{(s)}$ in \eqref{restrs}, we need to work a bit more. In Proposition \ref{P:Knapp} below we show that, given $p$ within the range \eqref{thick}, a necessary condition for \eqref{restrs} to hold is that
\begin{equation}\label{nec}
1\le q \le \frac{n+1-2s}{n+1} p'.
\end{equation}

Notice that when the target space is $L^2$, then in view of \eqref{nec} the conjecture asks whether it is true that \eqref{restrs} holds with $q=2$ for any $f\in \mathscr S(\Rn)$ and  
 $2 \le \frac{n+1-2s}{n+1} p'$. This is equivalent to asking that $\frac 1{p'} \le \frac{n+1-2s}{2(n+1)} = \frac 12 - \frac{s}{n+1}$, or equivalently $\frac 1p \ge \frac 12 + \frac{s}{n+1} = \frac{n+1+2s}{2(n+1)}$, and therefore
for any 
\begin{equation}\label{tss}
1\le p \le \frac{2(n+1)}{n+1+2s}.
\end{equation}

In the next result, we prove this conjecture.

\begin{theorem}\label{T:TSs}
For a given $s\in (0,1)$ let $p = \frac{2(n+1)}{n+1+2s}$. Then there exists a constant $C^{(s)}(n)>0$ such that for every $f\in \mathscr S(\Rn)$ one has
\begin{equation}\label{restr2}
\left(\int_{\Rn} |\hat f(\xi)|^2 A^{(s)}_1(\xi) d\xi\right)^{1/2} \le C^{(s)}(n)\ ||f||_{L^p(\Rn)}.
\end{equation} 
\end{theorem}

In order to establish Theorem \ref{T:TSs}
we exploit Plancherel as in \cite{To} and reduce matters to proving that the operator $R^{(s)} f = \widehat{A^{(s)}_1} \star f$ maps $L^p(\Rn)\to L^{p'}(\Rn)$. This is ultimately achieved using Stein's theorem of complex interpolation by embedding $R^{(s)}$ into an analytic family of operators $\{T_z\}_{z\in S}$ in the strip $S = \{z\in \mathbb C\mid - \frac{n-1}2\le \Re z\le 1\}$. Specifically, we show that 
\[
\begin{cases}
T_{-\frac{n-1}2 + i y}: L^1(\Rn)\ \longrightarrow\ L^\infty(\Rn),
\\
T_{1 + i y}: L^2(\Rn)\ \longrightarrow\ L^2(\Rn),
\end{cases}
\]
with appropriate bounds on the operator norms.   
Since the constant $C^{(s)}(n)$ in \eqref{restr2} is bounded uniformly in $s\in (0,1)$, see \eqref{Mss} below, passing to the limit as $s\to 1$ we recover the  Tomas-Stein restriction theorem, see \cite{To}.

\begin{remark}\label{R:4}
One should observe that the threshold exponent $p = \frac{2(n+1)}{n+1+2s}$ in \eqref{tss} is strictly less than $2$ for any $0<s<1$. Therefore, the limitation $\frac 12 \le s <1$ is not necessary in such case. 
\end{remark}

The plan of the paper is as follows. In Section \ref{S:nec} we adapt the well-known  argument of Knapp to prove that, if for a given $p$ in the range \eqref{thick} the restriction inequality \eqref{restrs} does hold, then we must have $1\le q \le \frac{n+1-2s}{n+1} p'$. Section \ref{S:proof} is devoted to proving Theorem \ref{T:friesz}, from which we obtain Corollary \ref{C:as}. The representation formula \eqref{infty} contained in it will be quite important in the remainder of the paper. Finally in Section \ref{S:TSs} we prove the nonlocal restriction Theorem \ref{T:TSs}. As a corollary of this result we obtain the celebrated Tomas-Stein theorem. The paper closes with an appendix in Section \ref{S:app} in which we gather some well-known facts and collect some  results needed in the rest of the paper. 

\medskip

\noindent \textbf{Acknowledgment:} I thank Carlos Kenig and Agnid Banerjee for their gracious preliminary reading of the manuscript and their feedback.  

%%%%%%%%%%%%%%%%%%%%%%%%%%%%%%%%%%%%%%%

\section{Necessary condition for restriction}\label{S:nec}

In \eqref{thick} above we have seen that the inequality \eqref{restrs} can possibly hold only when $1\le p < \frac{2n}{n+2s -1}$. In this section we adapt the well-known  argument of Knapp to prove that, if for a given $p$ in such range  \eqref{restrs} does hold, then we must have $1\le q \le \frac{n+1-2s}{n+1} p'$.

\begin{proposition}\label{P:Knapp}
A necessary condition for \eqref{restrs} to hold for $1\le p < \frac{2n}{n+2s -1}$ is that 
$1\le q \le \frac{n+1-2s}{n+1} p'$.
\end{proposition}

\begin{proof}
For $\ve>0$ small consider the parallelepiped $K_\ve = K'_\ve\times[1-\ve,1]$ with sides parallel to the coordinate axis whose projection onto $\R^{n-1}\times\{0\}$ is the $(n-1)$-dimensional cube $K'_\ve$ circumscribing the intersection of the hyperplane $x_n = 1-\ve$ with the unit sphere $\so$. If $\theta_\ve$ is the angle of aperture of the right circular cone obtained by projecting to the origin the points of the $(n-2)$-dimensional sphere obtained intersecting $\so\cap \{x_n=1-\ve\}$, from elementary trigonometry we have $\cos \theta_\ve = 1-\ve$, $\sin \theta_\ve = R(\ve) = \sqrt{\ve(2-\ve)}$, and therefore $K'_\ve = [-R(\ve),R(\ve)]^{n-1}$.

Denoting now with $B'(0,r) = \{x'\in \R^{n-1}\mid |x'|<r\}$, consider now the right-circular cylinders $C_\ve = B'(0,R(\ve))\times [1-\ve,1]$ and $C^\star_\ve = B'(0,\sqrt{n-1}\ R(\ve))\times [1-\ve,1]$.  A moment's thought reveals that 
\begin{equation}\label{mt}
C_\ve \cap (\Rn\setminus B(0,1)) \subset K_\ve \cap (\Rn\setminus B(0,1)) \subset C^\star_\ve \cap (\Rn\setminus B(0,1)).
\end{equation}
We note that
\[
C_\ve \cap (\Rn\setminus B(0,1)) = \{(x',x_n)\in \Rn\mid 1-\ve\le x_n \le 1,\ \sqrt{1-x_n^2}\le |x'|\le R(\ve)\}.
\]
As in Knapp's argument, if $\mathbf 1_E$ is the indicator function of a set $E\subset \Rn$, we now consider the function $f_\ve = \mathscr F^{-1}(\mathbf 1_{K_\ve})$, so that $\hat f_\ve = \mathbf 1_{K_\ve}$. As it is well-known
\[
f_\ve(\xi) = \hat{\mathbf 1}_{K_\ve}(\xi) = e^{i\pi (2-\ve)\xi_n}\ \frac{\sin(\pi \ve\xi_n)}{\pi \xi_n} \prod_{j=1}^{n-1} \frac{\sin(2\pi R(\ve)\xi_j)}{\pi \xi_j},
\]
and therefore $f_\ve\in L^p(\Rn)$ for any $p>1$ and moreover
\begin{equation}\label{knapp}
||f||_{L^p(\Rn)} \cong \ve^{\frac{n+1}{2p'}}.
\end{equation}
Next, we want to understand the asymptotic behaviour as $\ve\to 0^+$ of the quantity
\begin{equation}\label{knapps}
\left(\int_{\Rn} |\hat f_\ve(x)|^q A^{(s)}_1(x) dx \right)^{1/q} =  \left(\int_{K_\ve \cap (\Rn\setminus B(0,1))}  A^{(s)}_1(x) dx\right)^{1/q}.
\end{equation}
In view of the inclusions \eqref{mt} it suffices to understand the asymptotic behaviour of the right-hand side of \eqref{knapps} when the integral is performed on the set $C_\ve \cap (\Rn\setminus B(0,1))$. With this objective in mind, we obtain from Cavalieri's principle

\begin{align*}
& \left(\int_{C_\ve \cap (\Rn\setminus B(0,1))}  A^{(s)}_1(x) dx\right)^{1/q} 
 = \left(\int_{1-\ve}^1 \int_{\sqrt{1-t^2}\le |x'|\le R(\ve)} A^{(s)}_1(x',t) dx' dt\right)^{1/q}
\\
& = \left(\int_{1-\ve}^1 \int_{\sqrt{1-t^2}\le |x'|\le R(\ve)} \frac{dx'}{(|x'|^2 - (1-t^2))^{s} (|x'|^2 + t^2)^{\frac n2}} dt\right)^{1/q}
\\
& \cong \left(\int_{1-\ve}^1 \int_{\sqrt{1-t^2}\le |x'|\le R(\ve)} \frac{dx'}{(|x'|^2 - (1-t^2))^{s}} dt\right)^{1/q}
\\
& \cong \left(\int_{1-\ve}^1 \int_{\sqrt{1-t^2}}^{R(\ve)} \frac{\rho^{n-2} d\rho}{(\rho^2 - (1-t^2))^{s}} dt\right)^{1/q}.
\end{align*}
We now want to show that as $\ve\to 0^+$ we have
\begin{equation}\label{G}
G(\ve) = \int_{1-\ve}^1 \int_{\sqrt{1-t^2}}^{R(\ve)} \frac{\rho^{n-2} d\rho}{(\rho^2 - (1-t^2))^{s}} dt \cong \ve^{\frac{n+1}2-s}.
\end{equation} 
To see this we write $G(\ve) = \int_{1-\ve}^1 F(\ve,t) dt$, where 
\[
F(\ve,t) = \int_{\sqrt{1-t^2}}^{R(\ve)} \frac{\rho^{n-2} d\rho}{(\rho^2 - (1-t^2))^{s}}.
\]
The chain rule gives
\[
G'(\ve) = F(\ve,1-\ve) + \int_{1-\ve}^1 \frac{\p F}{\p \ve}(\ve,t) dt = \int_{1-\ve}^1 \frac{\p F}{\p \ve}(\ve,t) dt,
\]
since $F(\ve,1-\ve) = 0$. A simple computation gives 
\[
\frac{\p F}{\p \ve}(\ve,t) = \frac{R'(\ve) R(\ve)^{n-2}}{(R(\ve)^2 - (1-t^2))^{s}},
\]
therefore
\begin{align*}
G'(\ve) & = R'(\ve) R(\ve)^{n-2} \int_{1-\ve}^1 \frac{dt}{(R(\ve)^2 - (1-t^2))^{s}}  =  R'(\ve) R(\ve)^{n-2} \int_{1-\ve}^1 \frac{1}{(t^2 - (1-\ve)^2)^{s}} dt
\\
& \cong R'(\ve) R(\ve)^{n-2} \int_{1-\ve}^1 \frac{dt}{(t - (1-\ve))^{s}} \cong \ve^{1-s} R'(\ve) R(\ve)^{n-2}.
\end{align*}
Since $R(\ve)^{n-2} = (\ve(2-\ve))^{n-2} \cong \ve^{\frac n2 -1}$, and $R'(\ve) \cong \ve^{-\frac 12}$, we infer that  $G(\ve) \cong \ve^{\frac{n+1}2-s}$, which gives the desired conclusion \eqref{G}. In conclusion, we have shown that
\begin{equation}\label{boom}
\left(\int_{\Rn} |\hat f_\ve(\xi)|^q A^{(s)}_1(\xi) d\xi\right)^{1/q} \cong \ve^{(\frac{n+1}2-s)\frac 1q}.
\end{equation}
Combining \eqref{knapp} with \eqref{boom} we finally infer that a necessary condition for \eqref{restrs} to hold is
\[
(\frac{n+1}2-s)\frac 1q \ge \frac{n+1}{2p'}\ \Longleftrightarrow\ 1\le q \le \frac{n+1-2s}{n+1} p'.
\]

\end{proof}

%%%%%%%%%%%%%%%%%%%%%%%%%%%%%%%%%%%%%%%%%%%%%%%%%%%%%%%%%%%%%%

\section{The Fourier transform of the kernel $A^{(s)}_1$}\label{S:proof}

In this section we prove Theorem \ref{T:friesz}. Using Lemmas \ref{L:critical} and \ref{L:hyper12} in the Appendix, we establish a result which provides a key step in the proof of Theorem \ref{T:friesz}.

\begin{lemma}\label{L:crux}
For every $0<s<1$, $r>0$ and $\xi\in \Rn\setminus\{0\}$ one has
\begin{align*}
& \int_1^\infty \rho^{-\frac n2} (\rho^2 - 1)^{-s}  J_{\frac n2 -1}(2\pi r |\xi|\rho)\ d\rho 
\\ 
& = \G(1-s)(2\pi r|\xi|)^{\frac n2 - 1} \left\{\frac{ \G(s)}{2^{\frac n2} \G(\frac n2)}- \frac{1}{2^s\G(s)} \int_0^{2\pi r |\xi|} t^{s - \frac n2} J_{\frac n2 - 1 + s}(t) dt\right\}.
\end{align*}
\end{lemma}

\begin{proof}

To prove Lemma \ref{L:crux} we use Lemma \ref{L:critical} in which we  take
\begin{equation}\label{par}
\nu = \frac n2 - 1,\ \ \ \ \ \alpha =  -\frac n2 + 1 = - \nu,\ \ \ \ \ \beta = 1-s,\ \ \ \ c = 2 \pi r |\xi|.
\end{equation}
Since $s<1$, the condition $\Re \beta>0$ is guaranteed. Also, $\alpha+2\beta<\frac 72$ is equivalent to $\frac n2 + 2s > - \frac 12$, which is trivially satisfied. Furthermore, we have
\[
\frac{\alpha + \nu}2 = 0,\ \ \ \ \ \nu+ 1 = \frac n2,\ \ \ \ \ \ \frac{\alpha+\nu}2 + \beta = 1-s,\ \ \ \ 1-\beta-\frac{\alpha+\nu}2 = s.
\]
We thus have
\begin{align*}
& \frac{c^\nu\ \G(\beta) \G(1-\beta-\frac{\alpha+\nu}2)}{2^{\nu+1} \ \G(\nu+1) \G(1-\frac{\alpha+\nu}2)} \ _1F_2(\frac{\alpha+\nu}2; \nu+1,\frac{\alpha+\nu}2 + \beta; - (\frac c2)^2)
\\
& = (2\pi r|\xi|)^{\frac n2 - 1}  \frac{\G(1-s) \G(s)}{2^{\frac n2} \G(\frac n2)} \ _1F_2(0; \frac n2,1-s; - (\frac{2\pi|\xi|}2)^2) = (2\pi r|\xi|)^{\frac n2 - 1}  \frac{\G(1-s) \G(s)}{2^{\frac n2} \G(\frac n2)},
\end{align*}
since from \eqref{zeroF} we have $_1F_2(0; \frac n2,1-s; - (\frac{2\pi|\xi|}2)^2) \equiv 1$. We next want to write in a more convenient form the second hypergeometric function in the right-hand side of \eqref{magic} below. With the above choices \eqref{par}, we now have
\[
1-\beta = s,\ \ \ \ \  2-\beta - \frac{\alpha+\nu}2 = s+1,\ \ \ \ \ \ 2-\beta+ \frac{\nu-\alpha}2 = s+1 + \frac n2 - 1 = \frac n2 + s,
\]
and also
\[
\alpha+2\beta-3 = - \frac n2 - 2s,\ \ \ \ \ 2-\alpha-2\beta = \frac n2 - 1 + 2s,\ \ \ \ \ \beta + \frac{\alpha+\nu}2 - 1 = - s.
\]
This gives
\begin{align*}
& \frac{2^{\alpha+2\beta-3} c^{2-\alpha-2\beta} \G(\beta+ \frac{\alpha+\nu}2 -1)}{\G(2-\beta+\frac{\nu-\alpha}2)}  \ _1F_2(1-\beta; 2-\beta - \frac{\alpha+\nu}2,2-\beta+ \frac{\nu-\alpha}2; - (\frac c2)^2)
\\
& = \frac{2^{- \frac n2 - 2s} (2\pi r|\xi|)^{\frac n2 - 1 + 2s} \G(-s)}{\G(\frac n2 + s)} \ _1F_2(s; s+1,\frac n2 + s; - (\frac{2\pi r|\xi|}2)^2)
\\
& = - \frac{2^{- \frac n2 - 2s}(2\pi r|\xi|)^{\frac n2 - 1 + 2s} \G(1-s)}{s \G(\frac n2 + s)} \ _1F_2(s; s+1,\frac n2 + s; - (\frac{2\pi r|\xi|}2)^2)
\end{align*}
If we now apply Lemma \ref{L:hyper12} with 
\[
a = 2\pi r |\xi|,\ \ \ \ c = 1,\ \ \ \ \nu = \frac n2 - 1 + s,\ \ \ \ \frac{\alpha+\nu}2 = s\ \Longrightarrow\ \alpha = s - \frac n2 + 1,
\]
we obtain
\begin{align*}
&  _1F_2(s; s+1,\frac n2 + s; - (\frac{2\pi r|\xi|}2)^2) = 
\frac{2s 2^{\frac n2 - 1 + s}\G(\frac n2 + s)}{(2\pi r |\xi|)^{2 s}} \int_0^{2\pi r |\xi|} t^{s - \frac n2} J_{\frac n2 - 1 + s}(t) dt.
\end{align*}
Substituting in the above, and putting everything together, we find
\begin{align*}
& \int_1^\infty \rho^{-\frac n2} (\rho^2 - 1)^{-s}  J_{\frac n2 -1}(2\pi r |\xi|\rho)\ d\rho
= 
\\
& \G(1-s)(2\pi r|\xi|)^{\frac n2 - 1} \left\{\frac{ \G(s)}{2^{\frac n2} \G(\frac n2)}- \frac{1}{2^s} \int_0^{2\pi r |\xi|} t^{s - \frac n2} J_{\frac n2 - 1 + s}(t) dt\right\},
\end{align*}
which finally gives the desired conclusion.

\end{proof}

\begin{proof}[Proof of Theorem \ref{T:friesz}] We begin by observing that the left-hand side of \eqref{ftAs} coincides with the right-hand side when $\xi = 0$. For this, note that on one hand \eqref{L:uno} gives 
\[
\widehat{A^{(s)}_r}(0) = ||A^{(s)}_r||_{L^1(\Rn)} = 1.
\]
On the other, we apply Lemma \ref{L:EH} with 
\[
a=1,\ \ \ \  \ \mu = s - \frac n2,\ \ \ \ \   \nu = \frac n2 + s - 1.
\]
In  such case, we have $\mu < \frac 12$ and also $\mu + \nu = s - \frac n2 + \frac n2 + s - 1 = 2s - 1 > -1$, since $s>0$. Since $\frac{\nu-\mu+1}2 = \frac n2$, we thus obtain
\begin{equation}\label{EH}
\int_0^{\infty} t^{s - \frac n2} J_{\frac n2 - 1 + s}(t) dt = 2^{s-\frac n2} \frac{\G(s)}{\G(\frac n2)}.
\end{equation}
This shows that when $\xi = 0$ the right-hand side of \eqref{ftAs} becomes
\begin{equation}\label{EH0}
\frac{2^{\frac n2-s} \G(\frac n2)}{\G(s)}\int_{0}^\infty t^{s - \frac n2} J_{\frac n2 - 1 + s}(t) dt = 1,
\end{equation}
and therefore \eqref{ftAs} does hold in $\xi = 0$.

Let now $\xi\not= 0$ and recall the well-known  formula of Bochner for the Fourier transform of a spherically symmetric function
\begin{equation}\label{FB} 
\hat f(\xi) = \frac{2\pi}{|\xi|^{\frac n2-1}} \int_0^\infty r^{\frac n2} f^\star(r) J(2\pi r|\xi|) dr,
\end{equation}
see \cite[Theor. 40 on p. 69]{BC}.
Applying \eqref{FB} to \eqref{ar}, after a simple change of variable  we find
\begin{align}\label{crux}
\widehat{A^{(s)}_r}(\xi)  & = c(n,s) \frac{2\pi}{r^{\frac n2 - 1} |\xi|^{\frac n2 -1}}  \int_1^\infty \rho^{-\frac n2} (\rho^2 - 1)^{-s}  J_{\frac n2 -1}(2\pi r |\xi|\rho)\ d\rho.
\end{align}
At this point we substitute Lemma \ref{L:crux} in \eqref{crux}, obtaining
\begin{align}\label{cruxxi}
& \widehat{A^{(s)}_r}(\xi)
=   (2\pi)^{\frac n2}\ c(n,s)\G(1-s) \left\{\frac{ \G(s)}{2^{\frac n2} \G(\frac n2)}- \frac{1}{2^s} \int_0^{2\pi r |\xi|} t^{s - \frac n2} J_{\frac n2 - 1 + s}(t) dt\right\}
\\
& = \frac{\sigma_{n-1}}2 c(n,s)\ \G(1-s)\G(s) \left\{1 - \frac{2^{\frac n2-s} \G(\frac n2)}{\G(s)} \int_0^{2\pi r |\xi|} t^{s - \frac n2} J_{\frac n2 - 1 + s}(t) dt\right\}
\notag
\\
& = \frac{\sigma_{n-1}}2 c(n,s)\ \G(1-s)\G(s) \frac{2^{\frac n2-s} \G(\frac n2)}{\G(s)} \left\{\int_0^{\infty} t^{s - \frac n2} J_{\frac n2 - 1 + s}(t) dt - \int_0^{2\pi r |\xi|} t^{s - \frac n2} J_{\frac n2 - 1 + s}(t) dt\right\},
\notag
\end{align}
where in the last equality we have used \eqref{EH0}. Using \eqref{norm} in \eqref{cruxxi}, we finally obtain for every $0<s<1$ and $r>0$ 
\begin{equation}\label{friesz}
\widehat{A^{(s)}_r}(\xi) = \frac{2^{\frac n2-s} \G(\frac n2)}{\G(s)}\int_{2\pi r |\xi|}^\infty t^{s - \frac n2} J_{\frac n2 - 1 + s}(t) dt.
\end{equation}
This completes the proof of Theorem \ref{T:friesz}.

\end{proof}

For subsequent purposes, it will be important to have the following alternative representation of $\widehat{A^{(s)}_r}$.
\begin{corollary}\label{C:as}
Let $0<s<1$. For every $\xi\in \Rn$ we have
\begin{align}\label{infty}
\widehat{A^{(s)}_r}(\xi) & =  \frac{2^{\frac n2-s} \G(\frac n2)}{\G(s)}\left\{n \int_{2\pi r |\xi|}^\infty t^{s - \frac n2-1} J_{\frac n2 + s}(t) dt -\frac{1}{(2\pi r|\xi|)^{\frac n2 -s}} J_{\frac n2 + s}(2\pi r|\xi|)\right\}.
\end{align}
The identity \eqref{infty} implies, in particular, the existence of a universal $C(n,s)>0$, such that as $|\xi|\to \infty$
\begin{equation}\label{infty2}
\left|\widehat{A^{(s)}_1}(\xi)\right| \le \frac{C(n,s)}{|\xi|^{\frac{n+1}2 -s}}.
\end{equation}
\end{corollary}

\begin{proof}
The proof of \eqref{infty} follows from \eqref{friesz} by applying the recursive formula \eqref{rec} with $\nu = \frac n2 + s$ and integrating by parts. One has in fact
\begin{align*}
& \int_{2\pi r |\xi|}^\infty t^{s - \frac n2} J_{\frac n2 - 1 + s}(t) dt = \int_{2\pi r |\xi|}^\infty t^{- n} \frac{d}{dt}\left(t^{\frac n2+s} J_{\frac n2+ s}(t)\right) dt.
\end{align*}
Details are left to the interested reader, who should notice that, since $s>0$, we now have $s - \frac n2-1+ \frac n2 + s = 2s-1>-1$, and therefore the oscillatory integral $\int_{0}^\infty t^{s - \frac n2-1} J_{\frac n2 + s}(t) dt$ is convergent near $t=0$ (and can in fact be explicitly computed via Lemma \ref{L:EH}). To prove \eqref{infty2} it is enough to observe that when $|\xi|$ is sufficiently large, then by \eqref{jnuinfty2} we have for some universal constant $C>0$ and all $t \ge 2\pi |\xi|$,
\[
|J_{\frac n2  + s}(t)|\le \frac{C}{t^{1/2}}.
\]
We thus find for the first term in the right-hand side of \eqref{infty}
\[
\left|\int_{2\pi  |\xi|}^\infty t^{s - \frac n2-1} J_{\frac n2 + s}(t) dt\right| \le C \int_{2\pi  |\xi|}^\infty t^{s - \frac n2-\frac 32} dt = \frac{2C}{(n+1-2s)(2\pi |\xi|)^{\frac{n+1}2-s}}. 
\]
Since the second term can obviously be estimated in the same way, we are finished. 

\end{proof}

The next result provides the limiting value of $\widehat{A^{(s)}_1}(\xi)$ as $s\to 1$. It represents the counterpart on the Fourier transform side of Proposition \ref{P:vague}.

\begin{lemma}\label{L:limit}
Let $n\ge2$. Then for every $\xi\in \Rn$ one has 
\[
\underset{s\to 1}{\lim}\ \widehat{A^{(s)}_1}(\xi) = \frac{1}{\sigma_{n-1}} \widehat{d\sigma}(\xi) 
\]
\end{lemma}

\begin{proof}

Notice that for any $\xi\not= 0$ and $t>2\pi |\xi|$ we have from \eqref{jnuinfty} 
\[
|t^{s - \frac n2-1} J_{\frac n2 + s}(t)| \le C(n,\xi)\ t^{- \frac{n+1}2} \in L^1(2\pi|\xi|,\infty),
\]
for some constant $C(n,\xi)>0$ independent of $s\in (0,1)$. By Lebesgue dominated convergence we thus have 
\[
\underset{s\to 1}{\lim} \int_{2\pi  |\xi|}^\infty t^{s - \frac n2-1} J_{\frac n2 + s}(t) dt = \int_{2\pi  |\xi|}^\infty t^{- \frac n2} J_{\frac n2 + 1}(t) dt = - \int_{2\pi  |\xi|}^\infty \frac{d}{dt}\left(t^{- \frac n2} J_{\frac n2}(t)\right) dt.
\]
This observation and \eqref{infty} give
\[
\underset{s\to 1}{\lim}\ \widehat{A^{(s)}_1}(\xi)  =  2^{\frac n2-1} \G(\frac n2) \frac{1}{(2\pi |\xi|)^{\frac n2 -1}} J_{\frac n2-1}(2\pi |\xi|).
\]
In view of \eqref{FTsphere} above, we have reached the desired conclusion when $\xi\not= 0$. When instead $\xi = 0$ for the left-hand side of \eqref{ftAs} we obtain from \eqref{L:uno}
\[
\underset{s\to 1}{\lim}\ \widehat{A^{(s)}_1}(0) = \underset{s\to 1}{\lim}\ ||A^{(s)}_1||_{L^1(\Rn)}= 1.
\]
Convergence to the same limit of the right-hand side follows from \eqref{EH0}.
 
\end{proof}

%%%%%%%%%%%%%%%%%%%%%%%%%%%%%%%%%%%%%%%%%%%%%%%%%%%%%%%%%%%%%%

\section{Proof of Theorem \ref{T:TSs}}\label{S:TSs}

In this section we prove Theorem \ref{T:TSs}. The proof will be based on two central ideas: 1) To exploit the $L^2$ nature of the inequality \eqref{restr2} via the Plancherel theorem. This reduces considerations to proving that the  nonlocal Tomas-Stein operator $R^{(s)}$ in \eqref{TSs} below maps $L^p$ into $L^{p'}$; 2) To accomplish this step, we embed $R^{(s)}$ into an analytic family of operators $T_z$. For the latter we show that   
\[
\begin{cases}
T_{-\frac{n-1}2 + i y}: L^1(\Rn)\ \longrightarrow\ L^\infty(\Rn)
\\
\\
T_{1 + i y}: L^2(\Rn)\ \longrightarrow\ L^2(\Rn),
\end{cases}
\]
with appropriate bounds on the operator norms.

\begin{proof}[Proof of Theorem \ref{T:TSs}]
Similarly to \cite{To} we write for $f\in \mathscr S(\Rn)$
\begin{align*}
& \int_{\Rn} |\hat f(\xi)|^2 A^{(s)}_1(\xi) d\xi = \int_{\Rn} \overline{\hat f(\xi)} \hat f(\xi) A^{(s)}_1(\xi) d\xi
\\
& = \int_{\Rn} \overline{\hat f(\xi)}\ \hat f(\xi) A^{(s)}_1(\xi) d\xi =  \int_{\Rn} \overline{\hat f(\xi)}\ \widehat{R^{(s)} f}(\xi) d\xi,
\end{align*}
where we have defined 
\begin{equation}\label{TSs}
R^{(s)} f(\xi) = \widehat{A^{(s)}_1} \star f(\xi),
\end{equation}
so that $\widehat{R^{(s)} f}(\xi) = \hat f(\xi) A^{(s)}_1(\xi)$.
By Plancherel we thus obtain 
\begin{align*}
& \int_{\Rn} |\hat f(\xi)|^2 A^{(s)}_1(\xi) d\xi = \int_{\Rn} \overline{f(x)}\ R^{(s)} f(x) dx
\\
& \le ||f||_{L^p(\Rn)} ||R^{(s)} f||_{L^{p'}(\Rn)},
\end{align*}
where in the last inequality we have used H\"older. The proof will be finished if we can show that for every $0<s<1$ there exists $M_s = M_s(n)>0$ such that for $f\in \mathscr S(\Rn)$ one has
\begin{equation}\label{punchline}
||R^{(s)} f||_{L^{p'}(\Rn)} \le M_s\ ||f||_{L^p(\Rn)},
\end{equation}
for $p = \frac{2(n+1)}{n+1+2s}$. We want to accomplish \eqref{punchline} by interpolating between the two endpoints $L^1\to L^\infty$ and $L^2\to L^2$. This means we have to choose $\theta\in [0,1]$ such that
\[
\frac{n+1+2s}{2(n+1)} = \frac 1p = \frac{1-\theta}1 + \frac{\theta}2 = 1 - \frac{\theta}2,
\]
which gives 
\[
\frac{\theta}2 = \frac{1}{p'} = \frac{n+1-2s}{2(n+1)},
\]
and therefore
\begin{equation}\label{theta}
\theta = \theta(s) = \frac{n+1-2s}{n+1} = 1 - \frac{2s}{n+1},\ \ \ \ \ \ 1-\theta = \frac{2s}{n+1}.
\end{equation}
For $z\in \mathbb C$ such that $\Re z\le 1$ we define a linear operator $T_z:\mathscr S(\Rn)\to \mathbb C$ by letting
\begin{equation}\label{Tz}
\widehat{T_z f}(\xi) = A^{(1-z)}_1(\xi) \hat f(\xi),
\end{equation}
where for $\Re z \le 1$ we have let
\begin{equation}\label{az}
A^{(z)}_1(x) =   \frac{c(n,z)}{(|x|^2 - 1)_+^z |x|^n}, 
\end{equation}
with $c(n,z) = \frac{2}{\sigma_{n-1}\G(z)\G(1-z)}$. Notice that $c(n,z) = c(n,1-z)$ and that for $y\in \R$ we have 
\begin{equation}\label{igamma}
||A^{(i y)}_1||_{L^\infty(\Rn)} \le \frac{2}{\sigma_{n-1}|\G(i y)||\G(1-i y)|}.
\end{equation}
From Plancherel theorem, \eqref{Tz}, \eqref{igamma} and \eqref{gammiy} 
we conclude for some universal $C(n)>0$ and for every $y\in \R$
\begin{equation}\label{L2}
||T_{1+i y} f||_{L^2(\Rn)} \le M_1(y)\ ||f||_{L^2(\Rn)},
\end{equation}
with $M_1(y) = C(n)\ e^{\pi |y|}$.

We now introduce the kernels
\begin{equation}\label{Kz}
K_z(\xi) = \frac{2^{\frac n2-1+z} \G(\frac n2)}{\G(1-z)}\bigg\{n \int_{2\pi  |\xi|}^\infty t^{-z - \frac n2} J_{\frac n2 + 1-z}(t) dt -\frac{1}{(2\pi |\xi|)^{\frac n2 -1+z}} J_{\frac n2 + 1-z}(2\pi |\xi|)\bigg\}.
\end{equation}
Notice that, according to \eqref{infty} in Corollary \ref{C:as}, when $z=1-s$ we have $K_{1-s} = \widehat{A^{(s)}_1}$, and therefore \eqref{Tz} gives
\begin{equation}\label{Rs}
T_{1-s} f(\xi) = \widehat{A^{(s)}_1} \star f(\xi) = R^{(s)} f(\xi),
\end{equation}
where in the last equality we have used \eqref{TSs}. Also notice that, since by analytic continuation \eqref{infty} continues to be valid for any $z\in \mathbb C$ in the strip $0<\Re z < 1$, for any such $z$ we have from \eqref{Tz}
\begin{equation}\label{Tzz}
T_z f = K_z \star f.
\end{equation}
Since by \eqref{Kz} the kernel $K_z$ defines an analytic function of $z$ for $-\frac{n+1}2 < \Re z<1$ (see Lemma \ref{L:EH}), we can use \eqref{Tzz} to analytically extend the operator $T_z$ to the whole strip $S = \{z\in \mathbb C\mid -\frac{n-1}2<\Re z<1\}$. If we let $z = x + i y$, then we define 
\[
S_{x+iy} = T_{-\frac{n-1}2 + \frac{n+1}2 x+i y}.
\]
Note that $S_z$ is now defined on the strip $\Sigma =  \{z\in \mathbb C\mid 0<\Re z < 1\}$, and that \eqref{L2} now reads
\begin{equation}\label{L22}
||S_{1+i y} f||_{L^2(\Rn)} \le C(n)\ \sinh(\pi |y|)\ ||f||_{L^2(\Rn)}.
\end{equation}  
Since $S_{i y} = T_{-\frac{n-1}2 +i y}$, we next analyse the behaviour of $T_z$ on the line the line $L_0 = \{z\in \mathbb C\mid \Re z = - \frac{n-1}2\}$. Notice that \eqref{theta} gives
\[
(1-\theta)(-\frac{n-1}2) + \theta \cdot 1 = 1-s.
\]
Note that the equation $-\frac{n-1}2 + \frac{n+1}2 x = (1-s)$ gives 
\[
x = \frac{(1-s) + \frac{n-1}2}{\frac{n+1}2} = \frac{n+1-2s}{n+1} = 1 - \frac{2s}{n+1} = \theta,
\]
see \eqref{theta}. This shows that $S_\theta = T_{1-s}$.
In view of \eqref{L2} we conclude that, if we can show that $T_{-\frac{n-1}2 + i y} :L^1(\Rn) \to L^\infty(\Rn)$, with appropriate bounds on the operator norms, then by Stein's theorem of complex interpolation for an analytic family of operators, see \cite{Stein56} or \cite[Theor. 4.1, p. 205]{SW}, it will follow that $T_{1-s} : L^p(\Rn) \to L^{p'}(\Rn)$, as desired. 

In order to show that $T_{-\frac{n-1}2 + i y} :L^1(\Rn) \to L^\infty(\Rn)$ we will prove that $K_{-\frac{n-1}2 + i y}\in L^\infty(\Rn)$ and that moreover for some universal constant $C>0$ depending only on $n$, one has 
\begin{equation}\label{K00}
||K_{-\frac{n-1}2 + i y}||_{L^\infty(\Rn)} \le C e^{\frac{3\pi}2 |y|},\ \ \ \ \ \ \ \ y\in \R.
\end{equation}
From \eqref{Kz} we obtain
\begin{equation}\label{K0}
K_{-\frac{n-1}2 + i y}(\xi) = \frac{2^{-\frac{1}2 + i y} \G(\frac n2)}{\G(\frac{n+1}2 -i y)}\bigg\{n \int_{2\pi  |\xi|}^\infty t^{-\frac{1}2 - i y} J_{n+\frac 12 - i y}(t) dt -\frac{1}{(2\pi |\xi|)^{-\frac{1}2 + i y}} J_{n +\frac{1}2 - i y}(2\pi |\xi|)\bigg\}.
\end{equation}
Note that from \eqref{jnuinfty2} there exists a universal $R = R(n)>0$ such that when $|\xi| \ge R$ one has 
\begin{equation}\label{away}
\left|\frac{1}{(2\pi |\xi|)^{-\frac{1}2 + i y}} J_{n +\frac{1}2 - i y}(2\pi |\xi|)\right|\le 1.
\end{equation}
On the other hand, for $|\xi|\le R$ we have from \eqref{near}
\begin{align}\label{nearnow}
& \left|\frac{1}{(2\pi |\xi|)^{-\frac{1}2 + i y}} J_{n +\frac{1}2 - i y}(2\pi |\xi|)\right| \le \frac{\G(n+1)}{|\G(n+1-i y)|\ \G(n + \frac 12+1)} \frac{\left(2\pi |\xi|\right)^{n+1}}{2^{n+\frac 12}}
\\
& \le \frac{\sqrt 2 \pi^{n+\frac 12} \G(n) R^{n+1}}{\G(n+\frac 12)\ |\G(n+1-i y)|}.
\notag
\end{align} 
We now have for $n\ge 2$
\begin{align*}
|\G(n+1-i y)| & = |(n-i y)||\G(n-i y)|= |(n-i y)||(n-1-i y)|...|(1-i y)||\G(1-i y)|
\\
& \ge (1+y^2)^{\frac n2} |\G(1-i y)| = (1+y^2)^{\frac n2} \sqrt{\frac{\pi |y|}{\sinh \pi |y|}},
\end{align*}
where in the last equality we have used \eqref{gammiy}.
This gives
\begin{equation}\label{invgamma}
\frac{1}{|\G(n+1-i y)|}\le (1+y^2)^{-\frac n2} \sqrt{\frac{\sinh \pi |y|}{\pi |y|}}\le  \sqrt{\frac{\sinh \pi |y|}{\pi |y|}} \le  \frac 32 e^{\frac{\pi}2 |y|}.
\end{equation}
Inserting this information in \eqref{nearnow}, and combining the resulting estimate with \eqref{away}, we conclude that there exists $C(n)>0$ such that for every $\xi\in \Rn$ and any $y\in \R$ one has 
\begin{equation}\label{secpi}
\left|\frac{1}{(2\pi |\xi|)^{-\frac{1}2 + i y}} J_{n +\frac{1}2 - i y}(2\pi |\xi|)\right| \le C e^{\frac{\pi}2 |y|}.
\end{equation}
Next, we show that for some $C = C(n)>0$ one has for every $\gamma\in \R$
\begin{equation}\label{invgamma2}
\left|\frac{1}{\G(\frac{n+1}2 -i y)}\right| \le C e^{\frac{\pi}2 |y|}.
\end{equation}
To see this we apply Legendre duplication formula, see e.g. (1.2.3) in \cite{Le}, to write 
\[
2^{n-2i y}\G(\frac{n+1}2 -i y)\G(\frac n2+1 -i y) = \sqrt \pi \G(n+1 - 2i y).
\]
Using the estimate $|\G(z)|\le |\G(\Re z)|$, this gives
\begin{align*}
& \left|\frac{1}{\G(\frac{n+1}2 -i y)}\right| \le \frac{2^n |\G(\frac n2+1 -i y)|}{\sqrt \pi |\G(n+1 - 2i y)|} \le \frac{2^n \G(\frac n2+1)}{\sqrt \pi |\G(n+1 - 2i y)|} \le C(n) e^{\frac{\pi}2 |y|},
\end{align*}
where in the last inequality we have used \eqref{invgamma}. This proves \eqref{invgamma2}. Next, we show that for every $\xi\in \Rn$ and every $y\in \R$
\begin{equation}\label{firpi}
\left|\int_{2\pi  |\xi|}^\infty t^{-\frac{1}2 - i y} J_{n+\frac 12 - i y}(t) dt\right|\le C (1+|y|) e^{\frac{\pi}2  |y|}.
\end{equation}
To prove \eqref{firpi} observe that, in view of \eqref{jnuinfty} there exists $R>0$ depending on $n$ such that for $|\xi|\ge R$ we have for $t\in [2\pi |\xi|,\infty)$ 
\[
\left|J_{n+\frac 12 - i y}(t) - \sqrt{\frac2{\pi
t}}\cos\left(t-\frac{n+1}2 \pi + i \frac{\pi}2 y\right)\right|\le t^{-3/2}.
\]
Since 
\begin{align*}
& \cos\left(t-\frac{n+1}2 \pi + i \frac{\pi}2 y\right)= \cos(t+i \frac{\pi}2 y)\cos(\frac{n+1}2 \pi) + \sin(t+i \frac{\pi}2y)\sin(\frac{n+1}2 \pi)
\\
& = \begin{cases} (-1)^{n+1}\sin(t+i \frac{\pi}2\gamma),\ \ \ \ n\ \text{even},
\\
(-1)^{n}\cos(t+i \frac{\pi}2\gamma),\ \ \ \ n\ \text{odd},
\end{cases}
\end{align*}
 with obvious meaning of the notation, we infer
\begin{align*}
& \left|\int_{2\pi  |\xi|}^\infty t^{-\frac{1}2 - i\gamma} J_{n+\frac 12 - i\gamma}(t) dt\right|\le \left|\sqrt{\frac2{\pi
}}\int_{2\pi  |\xi|}^\infty \frac{\begin{pmatrix}\sin(t+i \frac{\pi}2\gamma)\\ \cos(t+i \frac{\pi}2\gamma)\end{pmatrix}}{t^{1 + i\gamma}} dt\right| +  \int_{2\pi  |\xi|}^\infty t^{-2} dt.
\end{align*}
Keeping in mind that for $z = x+iy$ we have
\[
\cos z = \cos x \cosh y - i \sin x \sinh y,\ \ \ \ \sin z = \sin x \cosh y + i \cos x \sinh y,
\]
and that integrating by parts we obtain
\[
\left|\int_{2\pi  |\xi|}^\infty \frac{\begin{pmatrix}\sin t\\ \cos t\end{pmatrix}}{t^{1 + i y}} dt\right| \le C \frac{1+|y|}{|\xi|}\le C (1+|y|),
\]
we conclude that \eqref{firpi} holds when $|\xi|\ge R$. If instead $|\xi|\le R$, then we write 
\begin{align*}
& \int_{2\pi  |\xi|}^\infty t^{-\frac{1}2 - i y} J_{n+\frac 12 - i y}(t) dt = \int_{0}^\infty t^{-\frac{1}2 - i y} J_{n+\frac 12 - i y}(t) dt - \int_0^{2\pi  |\xi|} t^{-\frac{1}2 - i y} J_{n+\frac 12 - i y}(t) dt,
\end{align*}
and then use Lemma \ref{L:EH} and \eqref{near} to estimate
\begin{align*}
& \left|\int_{2\pi  |\xi|}^\infty t^{-\frac{1}2 - i y} J_{n+\frac 12 - i y}(t) dt\right| \le \frac{|\G(\frac n2 - i y)|}{\sqrt 2 \G(\frac n2 + 1)} + \frac{C(n) R^{n+2}}{|\G(n-i y + \frac 12)|} \le C(n) \left(1 + \frac{1}{|\G(n-i y + \frac 12)|}\right).
\end{align*}
Using again Legendre duplication formula, similarly to the proof of \eqref{invgamma2} we recognise
\[
\frac{1}{|\G(n-i y + \frac 12)|} \le C(n) e^{\frac{\pi}2 |y|}.
\]  
If we now use \eqref{secpi}, \eqref{invgamma2} and \eqref{firpi} in \eqref{K0}, we conclude that \eqref{K00} does hold. 
It follows that for every $f\in \mathscr S(\Rn)$ and any $y\in \R$ one has
\begin{equation}\label{firstendpoint}
||T_{-\frac{n-1}2 + i y} f||_{L^\infty(\Rn)} \le M_0(y) ||f||_{L^1(\Rn)}, 
\end{equation}
with $M_0(y) = C e^{\frac{3\pi}{2}|y|}$. 
By \cite[Theor. 4.1 on p. 205]{SW} we infer that there exists $M_s = M_{\theta(s)}>0$ such that \eqref{punchline} holds. This proves Theorem \ref{T:TSs}.

\end{proof}

From p. 209 in \cite{SW} we see that the constant $M_s$ is given by
\[
M_s = \exp\left[\frac{\sin \pi \theta(s)}2 \int_{\R} \left\{\frac{\log M_0(y)}{\cosh \pi y - \cos \pi \theta(s)} + \frac{\log M_1(y)}{\cosh \pi y + \cos \pi \theta(s)}\right\} dy\right],
\]
From \eqref{theta} we have
\[
\sin \pi \theta(s) = \sin(\pi - \frac{2\pi s}{n+1}) = \sin(\frac{2\pi s}{n+1}),\ \ \ \ \ \ \cos \pi \theta(s)  = - \cos(\frac{2\pi s}{n+1}).
\]
We thus find
\begin{equation}\label{Mss}
M_s  = \exp\left[\sin(\frac{2\pi s}{n+1}) \int_0^\infty \left\{\frac{\log M_0(y)}{\cosh \pi y + \cos(\frac{2\pi s}{n+1})} + \frac{\log M_1(y)}{\cosh \pi y - \cos(\frac{2\pi s}{n+1})}\right\} dy\right].
\end{equation}
Since $\cos(\frac{2\pi}{n+1}) \not= \pm 1$ for any $n\ge 2$, by Lebesgue dominated convergence we find
\begin{align}\label{Ms}
& \underset{s\to 1}{\lim}\ M_s = \exp\left[\sin(\frac{2\pi}{n+1}) \int_0^\infty \left\{\frac{\log M_0(y)}{\cosh \pi y + \cos(\frac{2\pi}{n+1})} + \frac{\log M_1(y)}{\cosh \pi y - \cos(\frac{2\pi}{n+1})}\right\} dy\right]
\\
& = M(n)\ <\ \infty.
\notag
\end{align}

Since the function $s\to \frac{2(n+1)}{n+1+2s}$ is decreasing on $(0,1)$, with range $(\frac{2(n+1)}{n+3},2)$, if now $1\le p \le \frac{2(n+1)}{n+3}$, then for any $s\in (0,1)$ we also have $1\le p \le \frac{2(n+1)}{n+1+2s}$. From the proof of Theorem \ref{T:TSs} we infer
for every $f\in \mathscr S(\Rn)$ one has
\begin{equation}\label{restr22}
\left(\int_{\Rn} |\hat f(\xi)|^2 A^{(s)}_1(\xi) d\xi\right)^{1/2} \le \sqrt{M_s}\ ||f||_{L^p(\Rn)},
\end{equation} 
with $M_s$ give by \eqref{Mss} above. Passing to the limit for $s\to 1$ in \eqref{restr22}, and using Proposition \ref{P:vague} and \eqref{Ms}, we conclude the celebrated Tomas-Stein theorem for the sphere
\[
\left(\frac{1}{\sigma_{n-1}}\int_{\so} |\hat f(\xi)|^2  d\sigma(\xi)\right)^{1/2} \le \sqrt{M(n)}\ ||f||_{L^p(\Rn)}.
\]

%%%%%%%%%%%%%%%%%%%%%%%%%%%%%%%%%%%%%%%%%%%%

\section{Appendix: some classical material}\label{S:app}

In this section we collect some results needed in this paper. We recall the gamma function
\[
\G(z) = \int_0^\infty t^{z-1} e^{-t} dt,\ \ \ \ \ \ \ \ \ \Re z>0.
\]
Using the well-known formula $\G(z+1) = z \G(z)$ this function can be extended as a meromorphic function to the whole $\mathbb C$ having simple poles in $z = -k$, $k=0,1,...$, with residues $\operatorname{Res}(\G;-k) = \frac{(-1)^k}{k!}$. The reciprocal gamma function $1/\G(z)$ is an entire function with zeros at the negative integers. The following special values of such function will be useful subsequently, see e.g. formulas 6.1.29 and 6.1.31 on p. 256 in \cite{AS},
\begin{equation}\label{gammiy}
\frac{1}{|\G(\pm i y)|} = |y| \sqrt{\frac{\sinh(\pi |y|)}{\pi|y|}},\ \ \ \ \ \ \ \frac{1}{|\G(1\pm i y)|} = \sqrt{\frac{\sinh(\pi |y|)}{\pi|y|}}.
\end{equation}

Since this note is about the Fourier transform, Bessel functions (introduced and developed in the seminal work \cite{Bes}) play a pervasive role in it. We recall the series expansion of the Bessel function $J_\nu(z)$ for $|\arg z|<\pi$,
\begin{equation}\label{series}
J_\nu(z) = \left(\frac z2\right)^{\nu}\sum_{k=0}^\infty \frac{(-1)^k}{\G(\nu+k+1) k!} \left(\frac z2\right)^{2k},
\end{equation} 
see \cite{W}.
From \eqref{series} it is easy to obtain the following well-known formulas
\begin{equation}\label{rec}
\frac{d}{dz}\left[z^\nu J_\nu(z)\right] = z^\nu J_{\nu-1}(z),
\end{equation}
\begin{equation}\label{rec2}
- \frac{d}{dz}\left[z^{-\nu} J_\nu(z)\right] = z^{-\nu} J_{\nu+1}(z),
\end{equation}
and 
\begin{equation}\label{rec3}
\frac{2\nu}z J_\nu(z) - J_{\nu+1}(z) = J_{\nu-1}(z).
\end{equation}
For every $\nu\in \mathbb C$ such that $\Re\nu>-\dfrac12$ the function $J_\nu$ admits the following Poisson representation
\begin{equation}\label{extbessel}
J_\nu(z)
=\frac{1}{\G(\frac12)\G(\nu+\frac{1}{2})}\left(\frac{z}2\right)^\nu\int^1_{-1}
e^{izs}(1-s^2)^{\frac{2\nu-1}2}ds,
\end{equation}
where $z$ ranges in the complex plane cut along the negative real axis $(-\infty,0]$. When $\nu\in \no$ then we can take $z\in \mathbb C$. From \eqref{series}  one easily recognises that
\begin{equation}\label{bfbehzero}
J_\nu(z)\cong\frac{1}{\G(\nu+1)}\left(\frac{z}2\right)^\nu,\quad\text{as }z\to 0.
\end{equation}
Again from \eqref{extbessel} one sees that when $z\in \R$ and $z>0$, then for $\Re\nu>-\dfrac12$ one has
\begin{equation}\label{near}
|J_\nu(z)| \le \frac{\G(\Re \nu+\frac 12)}{|\G(\nu+\frac 12)|\ \G(\Re\nu +1)} \left(\frac{z}2\right)^{\Re\nu}.
\end{equation} 
This estimate is not useful for large $z>0$ since $J_\nu(z)$ decays at infinity with an oscillatory behaviour. When $\Re\nu>-\dfrac12$ the following result due to Hankel holds, see for instance \cite[Lemma 3.11]{SW}, or also (5.11.6) on p. 122 in \cite{Le}.
 One has for $0<\delta<\pi$
\begin{align}\label{jnuinfty} 
J_\nu(z)&=\sqrt{\frac2{\pi
z}}\cos\left(z-\frac{\pi\nu}2-\frac\pi4\right)+
O(z^{-\frac32})\\
&\quad\text{as }|z|\to\infty,\quad-\pi+\delta<\arg z<\pi-\delta.
\notag
\end{align}
In particular, 
\begin{equation}\label{jnuinfty2}
J_\nu(z) = O(z^{-1/2}),\ \ \ \ \ \text{as}\ z\to +\infty.
\end{equation}

The following beautiful formula can be found in 6.561.14 on p. 684 of \cite{GR}, or also (19) on p. 49 in vol.2 of \cite{EH}.

\begin{lemma}\label{L:EH}
Let $a>0$, $-\Re \nu-1<\Re\mu<\frac 12$. Then
\[
\int_0^\infty t^\mu J_\nu(at) dt = 2^\mu a^{-\mu-1} \frac{\G(\frac{\nu+\mu+1}2)}{\G(\frac{\nu-\mu+1}2)}.
\]
\end{lemma}

We next recall the definition of the hypergeometric functions. The Pochammer's symbols are defined by
\begin{equation}\label{poke}
\alpha_0 = 1,\ \ \ \ \alpha_k \overset{def}{=} \frac{\G(\alpha + k)}{\G(\alpha)} = \alpha(\alpha+1)...(\alpha + k -1),\ \ \ \ \ \ \ \ \ \ k\in \mathbb N. 
\end{equation}
Notice that since the gamma function has a pole in $z=0$, we have
\[
0_k = \begin{cases}
1\ \ \ \ \ \ \text{if}\ k = 0
\\
0\ \ \ \ \ \ \text{for}\ k\ge 1.
\end{cases}
\]
\begin{definition}\label{D:hyper}
Let $p, q\in \mathbb N\cup\{0\}$ be such that $p\le q+1$, and let $\alpha_1,...,\alpha_p$ and $\beta_1,...,\beta_q$ be given parameters such that $-\beta_j\not\in \no$ for $j=1,...,q$. Given a number $z\in \mathbb C$, the power series
\[
_p F_q(\alpha_1,...,\alpha_p;\beta_1,...,\beta_q;z) = \sum_{k=0}^\infty \frac{(\alpha_1)_k . . . (\alpha_p)_k}{(\beta_1)_k . . . (\beta_q)_k} \frac{z^k}{k!}
\]
is called the \emph{generalized hypergeometric function}. When $p = 2$ and $q=1$, then the function $_2 F_1(\alpha_1,\alpha_2;\beta_1;z)$ is the \emph{Gauss' hypergeometric function}, and it is usually denoted by $F(\alpha_1,\alpha_2;\beta_1;z)$.
\end{definition}

Using the ratio test one easily verifies that the radius of convergence of the above hypergeometric series is $\infty$ when $p\le q$, whereas it equals $1$ when $p = q+1$. Thus for instance it is $1$ for Gauss' hypergeometric function $F(\alpha_1,\alpha_2;\beta_1;z)$.  For later reference we record the following facts that follow easily from Definition \ref{D:hyper}:
\begin{equation}\label{zeroF}
_p F_q(0,\alpha_2,...,\alpha_p;\beta_1,...,\beta_q;z) = 1,
\end{equation}
and (see also p. 275 in \cite{Le})
\begin{equation}\label{fs6}
F(\alpha,\beta;\beta;-z) =\ _1F_0(\alpha;-z) = (1+z)^{-\alpha}.
\end{equation}

The following result plays a key role in this work. The interested reader can find it in formula 2.12.4.16 on p. 178 in \cite{PBM}.

\begin{lemma}\label{L:critical}
Let $c>0$, $\Re \beta>0$, $\Re(\alpha+2\beta)<\frac 72$. Then
\begin{align}\label{magic}
& \int_1^\infty \rho^{\alpha-1} (\rho^2 - 1)^{\beta-1} J_\nu(c \rho) d\rho 
\\
& = \frac{c^\nu\ \G(\beta) \G(1-\beta-\frac{\alpha+\nu}2)}{2^{\nu+1} \ \G(\nu+1) \G(1-\frac{\alpha+\nu}2)} \ _1F_2(\frac{\alpha+\nu}2; \nu+1,\frac{\alpha+\nu}2 + \beta; - (\frac c2)^2)
\notag\\
& + \frac{2^{\alpha+2\beta-3} c^{2-\alpha-2\beta} \G(\beta+ \frac{\alpha+\nu}2 -1)}{\G(2-\beta+\frac{\nu-\alpha}2)}  \ _1F_2(1-\beta; 2-\beta - \frac{\alpha+\nu}2,2-\beta+ \frac{\nu-\alpha}2; - (\frac c2)^2).
\notag
\end{align}
\end{lemma}

In order to fully exploit Lemma \ref{L:critical}, we next derive a useful representation formula of a certain hypergeometric function $_1 F_2$ in terms of an integral involving a Bessel function. We stress that, although for convenience we have used the same letters, the parameters $\alpha, \nu$ in the statement of the next result are not the same as those in Lemma \ref{L:critical}.

\begin{lemma}\label{L:hyper12}
Let $a, c>0$ and $\Re(\alpha+\nu)>0$. Then
\[
\ _1F_2\left(\frac{\alpha+\nu}2;\frac{\alpha+\nu}2+1,\nu+1; - (\frac{ac}{2})^2\right) = \frac{(\alpha+\nu) 2^{\nu}\G(\nu+1)}{a^{\alpha + \nu} c^\nu}\int_0^a t^{\alpha-1} J_\nu(ct) dt.
\]
\end{lemma}

\begin{proof}

Performing the change of variable $\tau = ct$ and using \eqref{series} we find
\begin{align*}
& \int_0^a t^{\alpha-1} J_\nu(ct) dt = c^{-\alpha}\int_0^{ac} \tau^{\alpha-1} J_\nu(\tau) d\tau
\\
& = 2^\alpha c^{-\alpha} \sum_{k=0}^\infty \frac{(-1)^k}{\G(\nu+k+1) k!} \int_0^{ac} \left(\frac \tau 2\right)^{2(k+\frac{\alpha+ \nu}2)} \frac{d\tau}\tau.
\end{align*}
The change of variable $z = \left(\frac \tau 2\right)^2$, for which $\frac{dz}z = 2 \frac{d\tau}\tau$, gives
\begin{align*}
&  \int_0^a t^{\alpha-1} J_\nu(ct) dt = 2^{\alpha-1} c^{-\alpha} \sum_{k=0}^\infty \frac{(-1)^k}{\G(\nu+k+1) k!} \int_0^{\left(\frac{ac}2\right)^2} z^{k+\frac{\alpha+ \nu}2} \frac{dz}z 
\\
& =  2^{\alpha-1} c^{-\alpha} \left(\frac{ac}2\right)^{\alpha+ \nu} \sum_{k=0}^\infty \frac{(-1)^k}{(\frac{\alpha+ \nu}2 + k)\G(\nu+k+1) k!} \left(\frac{ac}2\right)^{2k}
\\
& = \frac{a^{\alpha + \nu}c^\nu}{2^{\nu+1}}  \sum_{k=0}^\infty \frac{(\frac{\alpha+\nu}2)_k}{(\frac{\alpha+\nu}2 + 1)_k (\nu+1)_k  k!} \left(\frac{ac}2\right)^{2k} \frac{\G(\frac{\alpha+\nu}2)}{(\frac{\alpha+ \nu}2 + k)\G(\frac{\alpha+\nu}2+k)}\frac{\G(\frac{\alpha+\nu}2+k+1)}{\G(\frac{\alpha+\nu}2+1)}\frac{\G(\nu+k+1)}{\G(\nu+1)\G(\nu+k+1)}
\\
& =  \frac{a^{\alpha + \nu} c^\nu}{(\alpha+\nu) 2^{\nu}\G(\nu+1)}\ _1F_2\left(\frac{\alpha+\nu}2;\frac{\alpha+\nu}2+1,\nu+1; - (\frac{ac}{2})^2\right),
\end{align*}
where the details of the last equality are left to the reader.

\end{proof}

%%%%%%%%%%%%%%%%%%%%%%%%%%%%%%%%%%%%%
\bibliographystyle{amsplain}

\end{document}